\documentclass[12pt]{article}  % Changed from amsart to article for compatibility

\usepackage{amsmath, amssymb, amsthm}  % Core AMS packages
\usepackage{tikz-cd}                  % Included in TeX Live
\usepackage{xspace}
\usepackage{enumitem}
\usepackage{graphicx}
\usepackage{xcolor}
\usepackage[T1]{fontenc}

\newtheorem{theorem}{Theorem}[section]
\newtheorem{corollary}[theorem]{Corollary}
\newtheorem{lemma}[theorem]{Lemma}
\newtheorem{proposition}[theorem]{Proposition}

\newtheorem{mainthm}[theorem]{Main Theorem}

\theoremstyle{definition}
\newtheorem{definition}[theorem]{Definition}
\newtheorem{remark}[theorem]{Remark}

\numberwithin{equation}{section}

\newtheoremstyle{named}{}{}{\itshape}{}{\bfseries}{.}{.5em}{\thmnote{#3's }#1}
\theoremstyle{named}
\newtheorem*{namedconjecture}{Conjecture}

\newcommand{\HH}{\mathbb{H}}

\newcommand{\ZZ}{\mathbb{Z}}
\newcommand{\CC}{\mathbb{C}}
\newcommand{\RR}{\mathbb{R}}
\newcommand{\Poincare}{Poincar\'e\xspace}

\usepackage{setspace}
\usepackage[margin=1in]{geometry}  % Replaces custom layout settings

\begin{document}

\baselineskip=17pt  % Reasonable line spacing, single spaced

%%%%%%%%%%%%%%%%

\title{\textbf{Topological Rigidity of Quoric Manifolds}}

\author{
  Ioannis Gkeneralis\thanks{
    Department of Mathematics, Aristotle University of Thessaloniki, Thessaloniki, 54124 Greece. Email: \texttt{igkeneralis@math.auth.gr}}
  \and
  Stratos Prassidis\thanks{
    Department of Mathematics, University of the Aegean, Karlovassi, Samos, 83200 Greece. Email: \texttt{prasside@aegean.gr}}
}

\date{}  % Optional: set date or leave empty

\maketitle

\begin{center}
\end{center}

\section{Introduction}

 Rigidity theorems are central in Geometric Topology. The question is of how much the homotopy type of a manifold can determine its homeomorphism type. The basic conjecture that implies most (if not all) of the known rigidity results so far is the Farrell-Jones Isomorphism Conjecture. The conjecture deals with the non-equivariant case. But, when the manifolds are equipped with further structure (group action, stratification) there are not general, precise conjectures in this case. The Isomorphism Conjecture was inspired by the classical Borel Conjecture. In that direction there is a loosely stated conjecture formulated by F. Quinn.

\begin{namedconjecture}[Quinn]
Let $M^n$ be a manifold with a "nice" group action or a space with "nice" stratification. Let $N^n$ be another manifold that is equivariantly or stratified (respectively) homotopy equivalent to $M^n$. Then $N^n$ and $M^n$ are equivariantly or stratified homeomorphic, respectively.
\end{namedconjecture}

Our interest is on equivariant rigidity theorems. We consider groups that generalise the classical tori. More specifically, according to the classical Frobenius Theorem, there are three, finite dimensional associative $\RR$-algebras without zero divisors: the reals, the complex and the quaternions. The unit sphere in each of these spaces ($\ZZ/2\ZZ$, $S^1$, $S^3$ respectively) forms a group. Further, we define the corresponding torus ($(\ZZ/2\ZZ)^n$, $(S^1)^n$, $(S^3)^n$ respectively). Starting from a simplicial complex, for the first two cases there is a general construction, given in \cite{dj}, of natural spaces with the action of the group. More specifically, for the reals their methods are encoded (and generalized) in Vinberg's construction for Coxeter groups (\cite{Vinberg}) and for the complex they describe the construction of quasitoric manifolds with the action of the torus $T^n$. Quasitoric manifolds are the topological analogue of toric varieties (\cite{dj}, \cite{BuchPan2000}). The remaining quaternionic construction is given in \cite{ho}. There, starting from a simplicial complex, the construction of a space with an $(S^3)^n$ action is given. When the space is a manifold, it is called a quoric manifold. A variant of the equivariant rigidity question was settled for Coxeter groups in \cite{ps} and for tori in \cite{mp}.

For the remaining case, we generalize the methods used for the Coxeter groups and quasitoric manifolds. Let $Q^n=(S^3)^n$. We say that $Q^n$ acts on a manifold $M^{4n}$ locally regularly if, locally, the action is given by (quaternionic) multiplication or conjugation on each coordinate. Then the quotient is a manifold with corners. Conversely, starting with a manifold with corners and an appropriate function from its faces to the conjugacy classes of subgroups of $Q^n$, we can construct a locally regular (quoric) manifold. Our
main result is the following.

\begin{mainthm}[Rigidity of Quoric Manifolds] Let $M^{4n}$ be a closed locally regular quoric manifold over a nice $n$-manifold with corners $X$ which is a homotopy polytope i.e. all the faces of $X$ (and $X$ itself) are contractible manifolds with corners. Let $N^{4n}$ be a locally linear closed $Q^n$-manifold and $f:N^{4n} \to M^{4n}$ a $Q^n$-equivariant homotopy equivalence. Then $f$ is $Q^n$-homotopic to a $Q^n$-homeomorphism.
\end{mainthm}

The proof of the main theorem follows the methods of \cite{ps} and \cite{mp}.

\begin{itemize}
  \item We show that the action on $N^n$ is locally regular. For this result, we first prove that $N^n$ has the same isotropy groups as $M^n$ and $f$ is an isovariant homotopy equivalence. 
The proof that $N$ is locally regular depends on the action of $Q^n$ on the neighbourhood of a fixed point. That requires careful analysis of the representation of $Q^n$ on the tangent space of the fixed point.Technically it requires more analysis that the torus case since $Q^n$ is not commutative.

  \item Let $Y$ be the quotient manifold with corners of the action. We prove that $N^{4n}$ is $Q^n$-homeomorphic with the standard model constructed from $Y$. For this, there is an obstruction theory analogous to the torus case.
  \item The rest is standard. The map $f$ induces a face preserving homotopy equivalence $\phi: Y \to X$. Induction and standard surgery methods imply that $\phi$ is face homotopic to a face homeomorphism $\psi$. The map $\psi$ induces a $Q^n$-homeomorphism $g: N^n\to M^n$ that is homotopic to $f$.
\end{itemize}

It should be noted that there is an analogue of the quoric manifolds given as a generalization of quasitoric manifolds, namely \emph{Scott spaces} (\cite{scott2}, \cite{scott1}). In that approach, the free abelian group that appears in the torus case, is replaced by the free group as model. Unifying the construction, for the real, complex and quaternions, the presentation uses the properties of the quotient map as a stratified bundle \cite{quinn82} to compute the cohomology of the spaces involved. This approach suggests also a different related rigidity problem for quasitoric and quoric manifolds, namely, the stratified
version of the rigidity question. 

\begin{namedconjecture}
(Stratified Rigidity Question) Let $M$ be a Scott complex or quaternionic manifold with its natural stratification. If $N$ is a stratified space and $f:N\to M$ is a stratified homotopy equivalence, then $f$ is stratified homotopic to a stratified homeomorphism.
\end{namedconjecture}

In this paper (as in \cite{mp}) what is proved is the equivariant rigidity result. But both quasitoric and quoric manifolds come with a natural stratification,
so that the quotient map is a stratified bundle. So it is natural to ask if the stratified rigidity result is valid in this setting.

\section{Preliminaries}

Let $\HH$ be the space of quaternions and we denote by $Q$ the group of unit quaternions. Notice that $Q=S^3\cong SU(2)$. Since $Q$ is the double cover of the simple group $SO(3)$, the only endomorphisms of $Q$ are either trivial or inner automorphisms. Set $\phi_w:Q\to Q$, then 

$$\phi_w(s) = 
  \begin{cases}
  1 & \text{if  $w=0$} \\ 
  usu^{-1} & \text{if  $w=u$}
  \end{cases}$$
where $u$ is a unit quaternion.

The following are in \cite{ho}.

\begin{proposition} \label{prop:subgroupsIsoQ^n}
For subgroups of $Q^n$ we have 
\begin{itemize}
\item[a)] A subgroup of $Q^n$ is isomorphic to $Q$ if and only if it is of the form:
$$Q(u)=\lbrace (s_1,s_2,...,s_n)\in Q^n : s_i=\psi_{u_i}(t), t\in Q\rbrace$$
for some non-zero $u=(u_1,u_2,...,u_n)\in\mathbb{H}^n$, with each $u_i$ equal to zero or a unit quaternion.
\item[b)] A subgroup of $Q^n$ is isomorphic to $Q^k$ $(0\leq k \leq n)$, if and only if it is of the form
$$Q(u^1, ..., u^k) = Q(u^1)\cdot\cdot\cdot Q(u^k)$$
where $u^i \in \HH^n\setminus \lbrace 0\rbrace$, such that each component $u_i^j$ is zero or a unit quaternion, and for each $i=1,2,...,n$, there is at most one $j$ such that $u_i^j$ is non-zero.
\end{itemize}
\end{proposition}

For $u\in\HH^n$, define the \emph{characteristic set} of $u$ as
$$\gamma(u)= \lbrace i\in [n] \mid u_i \neq 0\rbrace$$
where $[n]=\lbrace 1,2,
\cdots n\rbrace$. Any two sets $\gamma_1, \gamma_2\subset [n]$ will be called \emph{compatible} if they are disjoint or if one is a subset of the other.

\begin{proposition}
The conjugacy classes of subgroups of $Q^n$ isomorphic to $Q^k$ for $k \leq n$ are classified by sets of $k$ disjoint characteristic sets $\gamma_i\subset [n]$, $j=1,...,k$,
denoted by $Q_{{\gamma}_1, \dots {\gamma}_k}$. \cite[p.17]{ho}
\end{proposition}

There is a natural action of $Q^n$ on $\HH^n$ given by coordinate-wise left multiplication,
$$((s_1,...,s_n),(h_1,...,h_n))\mapsto (s_1h_1,...,s_nh_n).$$
We call this action \emph{standard $Q^n$-action} on the \emph{standard $n$-corner} $\HH^n$. Standard actions satisfy the following general properties:

\begin{itemize}
\item[(1)] The orbits are products of $3$-spheres and the origin is the only fixed point.
\item[(2)] The orbit space $\HH^n/Q^n$ is the positive cone $\RR^n_{\geq}=\lbrace (h_1,...,h_n):h_i\geq 0\rbrace$, with radii $\mid h\mid$, $h\in \HH^n$ and the quotient map given by $(h_1,...,h_n)\mapsto (\mid h_1\mid^2,...,\mid h_n\mid^2)$.
\item[(3)] Conjugacy classes $Q^n_{x}$ are naturally associated with the faces of the orbits. We refer to these classes as \emph{isotropy classes}. 
\end{itemize}

Because of the lack of commutativity, the standard $Q^n$-action is very restrictive. We need a more general, although canonical action that includes the standard. Let $Q^n$ act on $\HH^n$ by $Q^n\times \HH^n \to \HH^n$,
\begin{equation} \label{eq:general_action}
((s_1,...,s_n),(h_1,...,h_n))\mapsto (s_{l_1}h_1s_{r_1}^{-1},...,s_{l_n}h_ns_{r_n}^{-1})
\end{equation} 
where each subscript is the index of some coordinate subgroup $Q_k$ (i.e. $Q_k=1\times\cdots\times 1\times Q\times 1\times\cdots 1$, with $Q$ in the k-th position) of $Q^n$ or is the empty set, with $s_{\emptyset}:=1$.

\begin{remark}
While Hopkinson in \cite{ho} focuses on quaternionic representations arising from coordinatewise multiplication by unit quaternions, more general $(S^3)^n$-actions also arise naturally in the study of quaternionic toric varieties. For instance, R. Scott in \cite{scott1} considers a broad class of such actions of $(S^3)^n$, defined via words in the free group on $n$ letters, leading to fibrations that go beyond just multiplication actions.  R. Scott in \cite{scott1}  gives explicit examples of such actions. These richer structures provide additional flexibility in constructing and classifying corners in the equivariant setting.
\end{remark}

In the general setting of equation (\ref{eq:general_action}), conditions $(1) - (3)$ may be violated, unless certain criteria are satisfied \cite[Chapter 3]{ho}. We now describe the criteria. 

Recall that a \emph{face} of the orbit space $\RR^n_{\geq}$ is a subset of points for which some subset of the coordinates is zero. A \emph{facet} $F_i$ is a face of codimension 1. Any face can be specified by the facets that contain it, 
$$F_{\sigma} = \bigcap_{i\in \sigma} F_i = \lbrace (r_1,...,r_n)\in \RR^n_{\geq} | r_i=0, i\in \sigma\rbrace$$ 
which defines the set $\sigma \subset [n]$, where $[n]$ is the set containing the first n natural numbers. 
The set $\mathcal{F}(\RR^n_{\geq})$ of all faces forms a partially ordered set (poset) with respect to the set inclusion. In fact, we define a category with the faces of $\RR^n$ as objects and the inclusions as morphisms, which we denote by $FACE(\RR^n_{\geq})$. Notice that 
$$F_{\sigma} \subset F_{\tau} \Longleftrightarrow \sigma \supset \tau.$$ 
Thus category (poset) $CAT([n])$ of all subsets of $[n]$ ordered by the inclusion can be identified with the category $FACE(\RR^n_{\geq})^{op}$.

For each $n\geq 0$, the set of all conjugacy classes of $Q^n$ isomorphic to $Q^k$ (some $0\leq k\leq n$) has a natural partial order derived from the subclass relation, so has the structure of a poset. Thus, we can define a category with these conjugacy classes as objects and subclass inclusion defining the morphisms, which we denote by $CONJ(Q^n)$. The action determines
a functor
$$\ell: CAT([n]) \to FACE(\RR^n_{\geq})^{op} \to CONJ(Q^n)$$
mapping each face to its isotropy class. 

\begin{definition}
A functor $\ell : CAT([n]) \to CONJ(Q^n)$ is an acceptable isotropy functor  if  $\ell$ is injective on the objects of $CAT([n])$.
\end{definition}

\begin{definition} \label{def:regular_corner}
Actions that determine an acceptable isotropy functor are called \emph{regular actions}. A space $\HH^n$ which is equipped with a regular $Q^n$-action is a \emph{regular $Q^n$-corner}.
\end{definition}

\begin{remark}
Any regular corner satisfies the general properties that we have identified in the standard action, namely (1) the orbits are products of $3$-spheres, (2) the orbit space is the product space 
$\RR^n_{\geq}$ and (3) the isotropy class of an orbit in a face $F_{\sigma}$ is a rank $|\sigma|$ class of subgroups in $Q^n$.
\end{remark}

\begin{remark}
In \cite{ho}) there are different characterizations of acceptability using either matrices or graphs. In \cite[p.60]{ho}) it is shown that all the characterizations are equivalent.
\end{remark}

\section{The canonical model}

In this section, we construct a quoric manifold from an n-manifold with corners $X$ and some linear data on the set of facets of $X$. This is analogous to the construction of the canonical model for torus actions \cite{BuchPan2000,bp,dj, mp2006}. 
\begin{definition}
Suppose \( Q^n \) acts on a manifold \( M^{4n} \). A \emph{regular chart} is a pair \( (U, \phi) \), where \( U \) is a \( Q^n \)-stable open subset of \( M^{4n} \), and \( \phi \) is an equivariant homeomorphism \( \phi: U \to \HH^n \) to a regular \( n \)-corner. That is, \( \phi \) satisfies 
\[
\phi(t \cdot y) = t \cdot \phi(y)
\]
for all \( t \in Q^n \), \( y \in U \subset M^{4n} \), where the action \( t \cdot \phi(y) \) is that of some regular corner (Definition \ref{def:regular_corner}).
A \( Q^n \)-action on \( M^{4n} \) is said to be \emph{locally regular} if is effective and \( M^{4n} \) has a regular atlas, that is, if every point of \( M^{4n} \) lies in a regular chart.
\end{definition}

The orbit space of a regular action is the positive cone $\RR^n_{\geq}$. Therefore, the orbit space of a locally regular action is a space locally modelled by $\RR^n_{\geq}$. The following definition formalizes this property.

\begin{definition}
A space $X$ is an \emph{n-manifold with corners} if it is Hausdorff, second countable space equipped with an atlas of open sets homeomorphic to open subsets of $\RR^n_{+}$ such that the overlap maps are local homeomorphisms that preserve the natural stratification of $\RR^n_{+}$ \cite{dav}.
\end{definition}

A manifold with corners $X^n$ is called \emph{nice} if
\begin{itemize}
\item[(1)] for every $0\leq k \leq n$ there is a codimension-$k$ face $F$,
\item[(2)] for each codimension-$k$ face $F$, there are exactly $k$ facets $F_1,...,F_k$ such that $F$ is a connected component of $F_1\cap\cdots\cap F_k$. Moreover $F$ does not intersect any other facet.
\end{itemize}

A nice manifold with corners $X^n$ is a \emph{homotopy polytope} if it is contractible as a manifold itself and each codimension-1 face (facet) is itself a contractible manifold with corners. The \emph{k-skeleton} of a manifold with corners $X^n$ is the set of all faces of dimension less than or equal to $k$ and it is denoted by $X^{(k)}$. \cite{BuchPan2000}

\begin{definition} \label{def:quoric}
Given a manifold with corners $X^n$, a \emph{quoric manifold} $M^{4n}$ over $X^n$ is a $Q^n$-manifold such that:
\begin{itemize}
\item[(1)] the action of $Q^n$ on $M^{4n}$ is locally regular,
\item[(2)] there is a continuous projection $\pi : M^{4n} \to X^n$, whose fibers are $Q^n$-orbits. In other words, the orbit space on $M^{4n}$ is homeomorphic to a manifold with corners $X^n$.
\end{itemize}
\end{definition}

Let $X^n$ be a manifold with corners. Any face of $X^n$ can be specified by the facets that contain it, $F_{\sigma}=\cap_{a\in \sigma} F_a$. The set of all $\sigma$ for which $F_{\sigma}$ is a face of $X^n$ is a poset and we denote it by $K_X$. Let $v=F_{\sigma_v}$ be a vertex of a manifold with corners $X$, where $\sigma_v=\lbrace a_1,...,a_n\rbrace$ in $K_X$. Restricting $CAT(K_X)$ to the vertex $v$, that is restricting to the subsets of $\sigma_v$, defines a full subcategory $CAT(K_X\mid_v)$ of $CAT(K_X)$, which is isomorphic to $CAT([n])\cong FACE(\RR^n_{\geq})$. So $r_v:CAT([n])\to CAT(K_X\mid_v)$ defined by $\sigma\mapsto\lbrace a_i\in \sigma_v\subset [n]\mid i\in\sigma \rbrace$ is an isomorphism. 

\begin{definition}
Let $X^n$ be a manifold with corners, a \emph{characteristic functor} over $X^n$ is a functor $\lambda : CAT(K_X) \to CONJ(Q^n)$ such that its restriction to any vertex $v$ of $X$ composed with $r_v$, is an acceptable isotropy functor,
$$l_v=\lambda\mid_v\circ r_v : CAT([n]) \overset{\cong}{\to} CAT(K_X\mid_v) \to CONJ(Q^n)$$
Also, a pair $(X^n,\lambda)$ is called a \emph{characteristic pair}.
\end{definition}

The next result is immediate.

\begin{proposition}
Any quoric manifold $M^{4n}$ over a manifold with corners $X^n$, has a characteristic functor over $X^n$. \cite[p.82]{ho}
\end{proposition}
 
\paragraph*{\textbf{Construction}}For $x\in X^n$, we denote by $F_{\tau(x)}$ the smallest face of $X^n$ that contains $x$ in its relative interior. Let $\hat{\lambda}(\tau(x))<Q^n$ denote the canonical subgroup 
$Q(u^1,...,u^n)$ in the conjugacy class $\lambda(\tau(x))=Q_{\gamma_1,...,\gamma_k}$, defined by $u^1,...,u^k \in \HH^n$ where $(u^i)_j=1$ if $j\in \gamma_i$ or $0$ otherwise. Define:
$$\mathcal{M}(\lambda) = Q^n \times X^n/\sim_{\lambda}$$
such that $(q,x)\sim_{\lambda}(q',x') \Longleftrightarrow x=x'$ and $q^{-1}q'\in \hat{\lambda}(\tau(x))$. The space $\mathcal{M}(\lambda)$ is a quoric manifold and $Q^n$ acts on it by acting on the first coordinate and the orbit space is $X^n$ \cite[p.86]{ho}.  It is called the \emph{canonical model} of the pair $(X^n, \lambda)$ or the \emph{derived quoric manifold}.

\begin{definition} 
Let $\phi: Y\to X$ be a map between n-manifolds with corners.
\begin{itemize}
\item[(1)] $\phi$ is called \emph{skeletal} if it preserves skeleta, i.e. $\phi (Y^{(k)})\subset X^{(k)}$.
\item[(2)] $\phi$ is called \emph{face-preserving} if for each face $F$ of $Y$, $\phi (F)$ is a face of $X^n$.
\end{itemize}
\end{definition}
Some natural properties of the above construction are summarised in the following remark. These results are implicit in \cite{mp}.
\begin{remark}
\begin{itemize}
\item[(1)] Starting with two quoric manifolds $N^{4n}$ and $M^{4n}$ over $Y$ and $X$, respectively, and a $Q^n$-equivariant homotopy equivalence between them, say $f: N^{4n}\to M^{4n}$, we get a face-preserving homotopy equivalence $\phi :Y\to X$ induced by $f$. \cite[Proposition 3.8]{mp}
\item[(2)] In addition, if $X$ is a homotopy polytope, then $Y$ is also a homotopy polytope. \cite[Corollary 3.9]{mp}
\item[(3)] The face-preserving homotopy equivalence $\phi$ induces an equivariant map $\phi_{\ast}:\mathcal{M}(\lambda')\to \mathcal{M}(\lambda)$. \cite[Proposition 3.10]{mp}
\end{itemize}
\end{remark}
 
Now, let $X$ be a nice n-manifold with corners and $F$ be a face of $X$, we denote with $Q_F$ the corresponding canonical group $Q(u^1, \dots,u^k)$.

The next result is the reverse construction of the above observations, which also appears in \cite[Proposition 3.13]{mp}.

\begin{proposition} \label{prop:reverse_construction}
Let $(X, \lambda)$ and $(Y,\lambda')$ be two characteristic pairs such that $X$ and $Y$ are homotopy polytopes. Let $f:\mathcal{M}(\lambda')\to\mathcal{M}(\lambda)$ be an equivariant homotopy equivalence. Then
\begin{itemize}
\item[(1)] the map $\phi:Y\to X$ induced on the quotients is a face-preserving homotopy equivalence,
\item[(2)] $Q_{F'}=Q_{\phi(F')}$ for each facet $F'$ of $Y$,
\item[(3)] there is an equivariant homotopy such that $f\simeq \phi_{\ast}$.
\end{itemize}
\end{proposition}

\section{On sections of orbit maps and Canonical Models}

Let $M^{4n}$ be a quoric manifold with orbit (quotient) map $\pi:M^{4n}\to X^n$. The $Q^n$-action determines a characteristic functor on $CAT(K_X)$ \cite{ho}. Let $\mathcal{M}(\lambda)$ be the canonical model associated to the pair $(X^n,\lambda)$. In order to compare $M^{4n}$ with its canonical model $\mathcal{M}(\lambda)$, we construct sections of the orbit (quotient) map $\pi:M^{4n}\to X^n$. As with the classical  bundle theory, we associate to the orbit map its Euler class, we prove that it has certain homotopy invariance properties, and its vanishing is equivalent to the existence of a section to the orbit map, which is always true in our case since the base point is contractible. 
This is a generalisation of the ideas analysed in \cite{Yoshida}. Our analysis follows closely \cite{Yoshida}. 

\begin{lemma} \label{lemma:Sections1}
For a characteristic pair $(X^n,\lambda)$, the projection map $\pi':\mathcal{M}(\lambda)\to X^n$ admits a section. 
\end{lemma}

\begin{proof}
Since $\mathcal{M}(\lambda)$ is a quoric manifold over $X^n$, the fibers of $\pi'$ are $Q^n$, so they are groups. Thus $\mathcal{M}(\lambda)$ admits the section $s': X^n\to \mathcal{M}(\lambda)$, $s'(p)=[1,p]$.
\end{proof}

\begin{proposition} \label{prop:hom2section}
Let $(X^n,\lambda)$ be a characteristic pair as above. Let $M^{4n}$ be a quoric manifold over $X^n$. Then $M^{4n}$ is $Q^n$-homeomorphic to $\mathcal{M}(\lambda)$ if and only if $\pi:M^{4n}\to X^n$ admits a section.
\end{proposition}
\begin{proof}
Let $h$ denote the $Q^n$-homeomorphism between $M^{4n}$ and $\mathcal{M}(\lambda)$. 
Define $s:X^n\to M^{4n}$ such that,
$$s(x)=h^{-1}\circ s'(x)$$
where $s'$ is defined in Lemma \ref{lemma:Sections1}, then $\pi\circ s(x) =\pi\circ h^{-1}\circ s'(x)=\pi'\circ s'(x)=id_P(x)=x$.

On the other hand, let $s:X^n\to M^{4n}$ be the section of $\pi$. Define the map $h:\mathcal{M}(\lambda)\to M^{4n}$ by $h(q,x)=q\cdot s(x)$ is obviously a $Q^n$-homeomorphism. 
\end{proof}

Notice that, by construction, the orbit map 
$$\pi_{\HH^n}:\HH^n \to \RR^n_{\geq}, \;\; \pi_{\HH^n}(h):=(\mid h_1\mid ^2,...,\mid h_n\mid ^2)$$
admits a section $i:\RR^n\to \HH^n$ defined by $i(h)=(h_1^{1/2},\cdots,h_n^{1/2})$. 

\begin{proposition}\label{prop:construction theta}
If $\pi : M^{4n} \to X^n$ has a section $s$ and $i$ is the section of $\pi_{\HH^n}$, then there exists a regular atlas of $M$,
$(U_{\alpha}^M, \phi_{\alpha}^M)$, $\alpha\in \mathcal{A}$ such that the following diagram commutes\\
\begin{center}
\begin{tikzcd}
U_{\alpha}^M \arrow[r, "\phi_{\alpha}^M"] 
& \phi_{\alpha}^M(U_{\alpha}^M)  \\
U_{\alpha}^X \arrow[r, "\phi_{\alpha}^X" ] \arrow[u, "s"]
& \phi_{\alpha}^X(U_{\alpha}^X) \arrow[u, "i"]
\end{tikzcd}
\end{center}
where $(U_{\alpha}^X,\phi_{\alpha}^X)$ is the atlas of $X^n$ induced by $(U_{\alpha}^M, \phi_{\alpha}^M)$.
\end{proposition}

\begin{proof}
The idea of the proof is first to construct an atlas of $X$ as orbits of a regular atlas on $M$. For the open sets of the atlas of $X$, we construct local sections of the corresponding 
derived model projection. We define the new atlas of $M$ with the desired properties using the original atlas, after twisting it with local sections. We provide with
 the explicit constructions for the reader's convenience.

Let  $(U_{\alpha}^M,\psi_{\alpha}^M)$ be a regular atlas of $M^{4n}$. First, we construct an atlas of $X^n$ as follows. We put $U_{\alpha}^X:=U_{\alpha}^M/Q^n$. Now, we restrict $\pi_{\HH^n}$ to $\psi_{\alpha}^M(U_{\alpha}^M)$, which induces a homeomorphism from $U_{\alpha}^X$ to the open subset $\pi_{\HH^n}(\psi_{\alpha}^M(U_{\alpha}^M)$ of $\RR^n_{+}$, which is denoted by $\psi^X_{\alpha}(U_{\alpha}^X)$. Thus $(U_{\alpha}^X,\psi_{\alpha}^X)$ is the atlas of $X^n$ induced by $(U_{\alpha}^M,\psi_{\alpha}^M)$. By construction, this means that the following diagram commutes\\
\begin{center}
\begin{tikzcd}
U_{\alpha}^M\subset M^{4n} \arrow[r, "\psi_{\alpha}^M"] \arrow[d, "\pi"]
& \psi_{\alpha}^M(U_{\alpha}^M)\subset \HH^n \arrow[d, "\pi_{\HH^n}"] \\
U_{\alpha}^X \subset X^n\arrow[r, "\psi_{\alpha}^X" ] 
& \psi_{\alpha}^X(U_{\alpha}^X)\subset \RR^n_{\geq} 
\end{tikzcd}
\end{center}

For each $\alpha\in\mathcal{A}$ and $b\in U_{\alpha}^X$, we have that $\pi\circ (\psi_a^M)^{-1}=(\psi_a^X)^{-1}\circ \pi_{\HH^n}$. Then $\pi\circ (\psi_a^M)^{-1} \circ i=(\psi_a^X)^{-1}\circ \pi_{\HH^n}\circ i=(\psi_a^X)^{-1}$. 
Thus for $b\in U_{\alpha}^X$, $\pi\circ (\psi_a^M)^{-1} \circ i\circ\psi_a^X(b) = b = \pi s(b)$. 
The equality 
$$\theta_{\alpha}(b)\cdot s(b)=(\psi_{\alpha}^M)^{-1}\circ i \circ \psi_{\alpha}^X (b)$$
for $b\in U_{\alpha}^X$ determines a local section in the derived quoric manifold $\mathcal{M}(\lambda)$, $\theta_{\alpha}:U_{\alpha}^X\to Q^n$.

Now, as in \cite[Proposition 5.1]{Yoshida}, we define a new coordinate system $\phi_{\alpha}^M$ on $U_{\alpha}^M$ by
$$\phi_{\alpha}^M(x):=\psi_{\alpha}^M(\theta_{\alpha}(\pi(x))\cdot x)$$
for $x\in \pi^{-1}(U_{\alpha}^X)$, which satisfies $\phi_{\alpha}^M(s(b))=\psi_{\alpha}^M(\theta_{\alpha}(b)\cdot s(b))=\psi_{\alpha}^M((\psi_{\alpha}^M)^{-1}\circ i\circ \psi_{\alpha}^X(b))=i\circ \psi_{\alpha}^X(b)$. In other words, the following diagram commutes
\begin{center}
\begin{tikzcd}
U_{\alpha}^M \arrow[r, "\phi_{\alpha}^M"] 
& \phi_{\alpha}^M(U_{\alpha}^M)  \\
U_{\alpha}^X \arrow[r, "\phi_{\alpha}^X" ] \arrow[u, "s"]
& \phi_{\alpha}^X(U_{\alpha}^X) \arrow[u, "i"]
\end{tikzcd}
\end{center}
where $\phi_{\alpha}^X=\psi_{\alpha}^X$. Therefore $(U_{\alpha}^M,\phi_{\alpha}^M)$ is the required regular atlas.
\end{proof}

As a consequence of the construction above, on each nonempty overlap $U_{\alpha\beta}^X=U_{\alpha}^X\cap U_{\beta}^X$, we define a family of functions $\check{\theta} =\lbrace \theta_{\alpha\beta}\rbrace$, which we will show that it forms a $\check{C}$ech one - chain on $\lbrace U_{\alpha}^X \rbrace$. For each quoric manifold $M^{4n}$, we have from the definition \ref{def:quoric} an open cover of $M^{4n}$: $\mathcal{U}=\lbrace U_{\alpha}^M \rbrace$. We assume that the index set $\mathcal{A}$, of the regular atlas, is countably ordered. Let $\pi(U_{\alpha}^M)=U_{\alpha}^X$ the induced open subset of $X$.
By Proposition \ref{prop:hom2section} there exist homeomorphisms $h_{\alpha}, h_{\beta}$ such that
$$h_{i}:U_{i}^X\times Q^n/\sim\to U_{i}^M, \quad (q,x) \mapsto q\cdot s_{i}(x), \quad i = \alpha, \beta$$

So for $x\in\pi^{-1}(U_{\alpha\beta}^X)$, $x$ and $h_{\alpha}\circ h_{\beta}^{-1}(x)$ belong to the same orbit. As in the proof of Proposition \ref{prop:construction theta}, this means that for any $b\in U_{\alpha\beta}^X$ the equation
$$h_{\alpha}\circ h_{\beta}^{-1}(x)=\theta_{\alpha\beta}^M(b)x, \quad \theta_{\alpha\beta}^M(b)\in Q^n, $$
determines a local section $\theta_{\alpha\beta}^M$ of $\pi: \mathcal{M}(\lambda)\to X^n$ on $U_{\alpha\beta}^X$.

Now, let $\mathcal{S}$ denote the sheaf of local sections of $\pi : \mathcal{M}(\lambda)\to X^n$. Then local sections $\theta_{\alpha\beta}^M $ form a $\check{C}$ech one - chain $\check{\theta}$ on $U_{\alpha}^X$ with values in $\mathcal{S}$. 

Notice that $\lbrace \theta_{\alpha\beta}^M \rbrace$ is a cocycle: Note that $\lbrace \theta_{\alpha\beta}^M \rbrace$ must satisfy some compatibility conditions on every intersection, so-called $\check{C}$ech cocycle conditions \cite{LawsonMichelson}, namely $\theta_{\alpha\alpha}(b)x=Id(x)$ and $\theta_{\alpha\beta}\theta_{\beta\gamma}\theta_{\gamma\alpha}(b)x=Id(x)$ on $U_{\alpha\beta\gamma}=U_{\alpha}\cap U_{\beta}\cap U_{\gamma}$. Both conditions are straightforward calculations using the relation $h_{\alpha}\circ h_{\beta}^{-1}(x)=\theta_{\alpha\beta}^M(b)x$. 

Let $\check{H^1}(X^n,\mathcal{S})$ denote the first $\check{C}$ech cohomology group of $X^n$ with values in $\mathcal{S}$. Since $\lbrace \theta_{\alpha\beta}^M \rbrace$ is a cocycle, it defines a cohomology class in $\check{H^1}(X^n,\mathcal{S})$. We denote it by $e(M)$.

\begin{definition}
We call $e(M)$ the \emph{Euler class} of $\pi: M^{4n}\to X^n$.
\end{definition}

\begin{theorem} \label{thm:EulerClass}
The orbit map $\pi : M^{4n}\to X^n$ has a section if and only if $e(M)$ vanishes. 
\end{theorem}
\begin{proof}
$e(M)$ vanishes means that a representative $\check{\theta}\in e(M)$ is cohomologous to the zero cocycle. So there exists $0$ - cocycle $\theta$, such that $\delta\theta=\check{\theta}$. But, here, cocycles $\theta\in \check{H^0}(X^n,\mathcal{S})$ are cochains $\lbrace \theta_{\alpha} \rbrace$ which satisfy 
\begin{equation} \label{eq:theta}
\theta_{\alpha}=\check{\theta}_{\alpha\beta}\theta_{\beta}
\end{equation}
for all $x\in U_{\alpha\beta}$. The equation (\ref{eq:theta}) becomes
$$\theta_{\alpha}(b)x=\theta_{\alpha\beta}^M\theta_{\beta}(b)x=h_{\alpha}\circ h_{\beta}^{-1}(x)\theta_{\beta}(b)x \Leftrightarrow h_{\alpha}^{-1}(\theta_{\alpha}(b)x)=h_{\beta}^{-1}(\theta_{\beta}(b)x)$$
for any $b\in U_{\alpha\beta}^X$, where we used the fact that $h_{\alpha}$ and $h_{\beta}$ are equivariant maps. In other words, $h_{\alpha}$ and $h_{\beta}$ agree on each non empty overlap $U_{\alpha\beta}$. This means that the map $h: M^{4n}\to \mathcal{M}(\lambda)$ defined by $h(x)=h_{\alpha}^{-1}(\theta_{\alpha}(b)x)$ for $x\in U_{\alpha}^M$ and $\pi(x)=b\in U_{\alpha}^X$ is $Q^n$-homeomorphism if and only if $e(M)$ vanishes.

So if the orbit map $\pi : M^{4n}\to X^n$ has a section, by Proposition \ref{prop:hom2section} $M^{4n}$ and $\mathcal{M}(\lambda)$ are $Q^n$-homeomorphic and thus $e(M)=0$.

On the other hand, if the Euler class $e(M)=0$ then we have that $M^{4n}$ and $\mathcal{M}(\lambda)$ are $Q^n$-homeomorphic. Again by Proposition \ref{prop:hom2section} the orbit map $\pi : M^{4n}\to X^n$ has a section. This completes the proof.
\end{proof}

\begin{corollary}
Let $M^{4n}$ be a quoric manifold over $X^n$ and $\mathcal{M}(\lambda)$ the canonical model associated to the action. Then the following are equivalent:
\begin{itemize}
\item[1.] There is a $Q^n$-homeomorphism $h:M^{4n}\to \mathcal{M}(\lambda)$.
\item[2.] The orbit map $\pi : M^{4n}\to X^n$ admits a section.
\item[3.] The Euler class $e(M^{4n})\in\check{H^1}(X^n,\mathcal{S})$ vanishes.
\end{itemize}
\end{corollary}

Now, we show that $e(M)$ has certain homotopy invariance properties. Intuitively, having quoric manifolds that have the same homotopy type, means that their cocycles are cohomologous. This comes to the same thing as showing that the induced homomorphisms are the same. The homotopy invariance theorem can be stated as follows.

\begin{theorem} \label{thm:HomInvarCech}
(Homotopy Invariance of $\check{C}$ech Cohomoly) Let $F: Y{\times}I \to Y$ be a homotopy between $f$ and $g$.
% $f$ and $g$ be maps of $Y$ into $Y\times I$ defined by setting $f(x)=(x,0)$ and $g(x)=(x,1)$ for all $x\in Y$. 
Then the induced homomorphisms $f_*$ and $g_*$ of $\check{\mathrm{H}}^p(Y)$ into $\check{\mathrm{H}}^p(Y\times I)$ are equal for each $p$.
\end{theorem}

The proof is omitted since it bears resemblances to the results that appear in \cite{Wallace}.

\begin{corollary} \label{cor:topclas2}
If $X$ is contractible then the orbit map $\pi$ admits a section and thus $M^{4n}$ is $Q^n$-homeomorphic to $\mathcal{M}(\lambda)$.
\end{corollary}

\begin{proof}
Since $X$ is contractible, $\check{\mathrm{H}}^1(X, \mathcal{S}) = 0$. The result follows.
\end{proof}

\begin{remark}
\begin{enumerate}
\item The corollary for quasitoric manifolds was proved in \cite[Proposition 1.8]{dj} (also \cite[Lemma 4.5]{mp2006}).
\item The proof of Corollary \ref{cor:topclas2} is very general. It is possible that there is a more direct proof.
\end{enumerate}
\end{remark}

\section{On $Q^n$-representations with standard isotropy subgroups}
In this section, we consider $Q = SU(2)$ and study $Q^n$-representations whose isotropy subgroups are \emph{standard}:

\begin{definition}\label{def:standard_isotropies}
Let $\rho$ be a representation of $Q^n$. An isotropy subgroup of $\rho$ is \emph{standard} if it is isomorphic to $Q^m$ for some $1\leq m\leq n-1$. It is \emph{maximal} if $m=n-1$.
\end{definition}

The complex representation theory of $Q = SU(2)$ is classical (see \cite{tomDieck,ItzRoStr1990,Hall}). All irreducible complex representations are either real or quaternionic \cite[p.286]{ItzRoStr1990}; odd-dimensional ones are of real type, even-dimensional ones quaternionic \cite[p.288]{ItzRoStr1990}. These correspond to structure maps discussed in \cite[Section II.6]{tomDieck}.

Let $K = \RR$, $\CC$, or $\HH$, and write $Rep(Q,K)$ for $Q$-representations over $K$, and $Irr(Q,K)$ for the irreducibles. For real representations, there is a canonical decomposition \cite[p.291]{ItzRoStr1990}:
\[
Rep(Q^n,\RR) = Rep_+(Q^n,\CC) \sqcup Rep_-(Q^n,\CC),
\]
where $Rep_+$ and $Rep_-$ are complexifications of real type and realifications of quaternionic type, respectively. For both cases see details in \cite[p.94]{tomDieck}. The isotropy groups coincide between real and corresponding complex forms. 

\begin{lemma}\label{lemma:Iso(p)=Iso(p_CC)}
Let $\rho \in Rep(Q^n,\RR)$. If $\rho$ arises from $\rho^+ \in Rep_+(Q^n,\CC)$, then $Iso(\rho) = Iso(\rho^+)$; likewise, if from $\rho_{\RR} \in Rep_-(Q^n,\CC)$, then $Iso(\rho) = Iso(\rho_{\RR})$.
\end{lemma}

We next examine isotropy in higher-dimensional irreducibles. The following Lemma is obvious.

\begin{lemma}\label{lemma:not_standard_isotropies}
    Let $\rho$ be an irreducible complex representation of $Q$ with $dim_{\CC}\rho=m+1>2$, then 
    \begin{itemize}
        \item[1.] the matrices that have $z=z_1^m, m>1$ as an eigenvector are of the form 
$$A=\left\{\begin{pmatrix}
a & 0\\
0 & \bar{a}
\end{pmatrix}, a=e^{\frac{ik\pi}{n}}, k\in\ZZ\right\} $$
with eigenvalue $e^{in\theta}$
        \item[2.] the non-trivial element $z=z_1^m$ has isotropy subgroup isomorphic to $\ZZ_m$.
     \end{itemize}
\end{lemma}

\begin{proposition} \label{prop:tensors_abelian}
Let $\rho=\rho_1\otimes\cdots\otimes\rho_k$ be a complex representation of $Q^m$, where each $\rho_i$ is an irreducible complex representation of $Q$. Assume that $dim_{\CC}\rho_i=d_i>2$, for $1\leq i\leq r$ and $\rho_i$ is the trivial 1-dimensional representation, for $r+1\leq i\leq k$. Then the element $z=z_1^{d_1}\otimes\cdots\otimes z_1^{d_r}\otimes z_1\otimes\cdots\otimes z_1$ has non-standard isotropy subgroup.
\end{proposition}
\begin{proof}

The stabilizer of $z$ includes the subgroup
\[
D_r = \left\{ (A_1, \ldots, A_r) \in Q^r : A_i z_1^{d_i} = \lambda_i z_1^{d_i},\ \prod \lambda_i = 1 \right\},
\]
which is abelian. Hence, the isotropy is $D_r \times Q^{m-r}$, not of standard form.
\end{proof}

\begin{corollary} \label{cor:dimension2}
Let $\rho=\rho_1\otimes\cdots\otimes\rho_k$ be a complex representation of $Q^m$, where each $\rho_i$ is an irreducible complex representation of $Q$ and $\rho$ has standard isotropy subgroups. Then each non-trivial $\rho_i$ has dimension $2$.
\end{corollary}
\begin{proof}
By Proposition \ref{prop:tensors_abelian}, higher-dimensional $\rho_i$ yield non-standard isotropy, which  implies that $dim_{\CC}\rho_i\leq 2$. Since $\rho_i$ is non-trivial, $\dim_\CC \rho_i \neq 1$, so it must be 2.
\end{proof}

Since the non-trivial irreducible complex representations of $Q$ are of dimension 2, we naturally turn to the question of how many non-trivial irreducible complex representations exist in the tensor product representation $\rho$.

\begin{proposition} \label{prop:at_most2}
Let $\rho=\rho_1\otimes\cdots\otimes\rho_k$ be a complex representation of $Q^m$, such that $\rho_i$'s are complex irreducible representations and $\rho$ has only standard isotropy groups, then at most two of the $\rho_i$'s are non-trivial.
\end{proposition}
\begin{proof}
By Corollary \ref{cor:dimension2}, the non-trivial representations are of dimension $2$.

First, we assume that there is $0\leq i\leq m$, such that $\rho_j$ is trivial for $j\neq i$. Thus we have that 
$$\rho=1\otimes\cdots \otimes 1\otimes\rho_i\otimes 1\otimes\cdots \otimes 1$$ 
Assume that $dim_{\CC}\rho_i=d_i>2$. For the element $z=z_1\otimes\cdots\otimes z_i\otimes\cdots\otimes z_m$, where $z_i=z_1^{d_i}$, the isotropy subgroup of z is 
$$Q_z=Q^{m-1}\times Q_{z_i}$$
By Lemma \ref{lemma:not_standard_isotropies} the isotropy subgroup of $Q_{z_i}$ is isomorphic to $\ZZ_{d_i}$. So it turns out that  
$$Q_z\simeq Q^{m-1}\times \ZZ_{d_i}$$
This is a contradiction, because the isotropy subgroups of $\rho$ are standard. Therefore, we get $dim_{\CC}(\rho_i)=2$.

Similarly, let $0\leq k,l\leq m$, such that $\rho_j$ is trivial for $j\neq k,l$, then
$$\rho=1\otimes \cdots \otimes 1\otimes\rho_l\otimes\rho_k\otimes 1\otimes\cdots \otimes 1$$ 
with $dim_{\CC}\rho_k=d_k$ and $dim_{\CC}\rho_l=d_l$. Assume that both representations $\rho_l$ and $\rho_k$ have dimension bigger than 2. By Proposition \ref{prop:tensors_abelian}, the isotropy group of the tensor product representation $\rho_k\otimes\rho_l$ , of the element $z_1^{d_k}\otimes z_1^{d_l}$ is 
$$Q_{kl}=\left\{ (A_{\lambda_k}(a_k),A_{\lambda_l}(a_l)): \lambda_k=\lambda_l^{-1} \right\} \subset S^1\times S^1$$
In this case, the isotropy group of the element $z=z_1\otimes\cdots\otimes z_1\otimes z_1^{d_k}\otimes z_1^{d_l}\otimes z_1\otimes\cdots\otimes z_1$ in $V_m$ is
$$Q_z\simeq Q^{m-2}\times Q_{kl}$$
which is not standard.

Now assume that $d_l=2$ and $d_k>2$. Again, we take the element $z_1\otimes\cdots\otimes z_1^{d_k}\otimes z_1^{d_l}\otimes\cdots\otimes z_1$, then by Lemma \ref{lemma:not_standard_isotropies} the isotropy group is isomorphic to
$$Q^{m-2}\times\ZZ_{d_k}$$
which is also not standard.

It turns out that the only case for the isotropy group of $\rho$ to be standard is when we have $dim_{\CC}(\rho_k)=dim_{\CC}(\rho_l)=2$.

Finally, let $d_i=dim_{\CC}\rho_i=2$, for $i=1,2,\cdots,r$, $r\geq 3$ and $dim_{\CC}\rho_j=1$, for $j\neq i$. We can assume 
$$\rho=\rho_1\otimes\rho_2\otimes\cdots\otimes\rho_r\otimes 1\otimes\cdots\otimes 1$$

Now, consider the element $z=z_1^2\otimes z_1^2\otimes\cdots\otimes z_1^2\otimes z_1\otimes\cdots\otimes z_1$. By Proposition \ref{prop:tensors_abelian}, the isotropy group of $\rho$ is
$$Q_z\simeq Q_{1\cdots r}\times Q^{m-r}$$
where $Q_{1\cdots r}=\left\{(A_{\lambda_1}(a_1),\cdots,A_{\lambda_r}(a_r)):  \lambda_1\cdots\lambda_r=1\right\}$ and $\lambda_i$'s are the eigenvalues of the Lemma \ref{lemma:not_standard_isotropies}. This means that $Q_{1\cdots r}$ is not standard.
\end{proof}

Let $Q^{n-1}_i$ denote the isotropy subgroup of $Q^n$, such that $Q^{n-1}_i=Q\times\cdots Q\times 1 \times Q\times\cdots\times Q$, with the trivial isotropy group in the i-th position.

\begin{corollary} \label{cor:conjugation isotropies}
Let $\rho=\bigoplus\rho_i$, $1\leq i\leq k$, in $Rep(Q^n,\CC)$ with standard isotropy subgroups. If $Q^{n-1}_i$ is maximal standard isotropy subgroup of $\rho$, then there is $\rho_j$ for $1\leq j\leq k$, such that $Q^{n-1}_i$ is the maximal isotropy subgroup of $\rho_j$.
\end{corollary}
\begin{proof}
Let $\rho=\bigoplus\rho_i$, $1\leq i\leq k$, in $Rep(Q^n,\CC)$. The isotropy group of $\rho$ is given by
$$Iso(\rho)=\bigcap\limits_{1 \leq j \leq k} Iso(\rho_j)$$
Consider a maximal standard isotropy group $Q^{n-1}_i$ of $\rho$. By Proposition \ref{prop:at_most2} the representations $\rho_j$'s are trivial except at most two of them. The result follows.
\end{proof}

In the next result, we show that a real $4$n-dimensional representation of $Q^n$ decomposes into a direct sum of irreducible $4$-dimensional representations, either of real or quaternionic type.

\begin{theorem} \label{thm:isotropy gps}
Let $\rho$ be a real representation of $Q^n$, with $dim_{\RR}\rho=4n$. Assume that all the isotropy subgroups of $\rho$ are standard. Also assume that there is a maximal isotropy subgroup which is conjugate to a unique $Q^{n-1}_i$, for each $i$. Then $\rho=\bigoplus\rho_i$, $1\leq i\leq k$, with $\rho_i$: irreducible representations such that $\rho_i=\rho_{i_1}\otimes\cdots\otimes\rho_{i_n}$, where $\rho_{i_j}$'s are trivial except at most two of them.
\end{theorem}
\begin{proof}
Let $\rho\in Rep(Q^n,\RR)$. We know that 
$$Rep(Q^n,\RR)=Rep^{+}(Q^n,\CC)\sqcup Rep^{-}(Q^n,\CC)$$
\cite[p.291]{ItzRoStr1990}.
First, assume that $\rho$ comes from $Rep^+(Q^n,\CC)$. Let $e_{+}(\rho)=\rho^{+}$, where $e_{+}$ denotes the process of complexification. That is, we take the real representation $\rho$ and complexify it \cite[p.94]{tomDieck}.

Consider a decomposition $\rho^{+}=\rho^{+}_1\oplus\cdots\oplus\rho^{+}_k$ of $\rho^{+}$ into non-trivial irreducible representations $\rho^{+}_i$ and an element $z$ in $V_j$, such that the isotropy subgroup of $z$, $Q_z=G$ is non-standard. This means that for an element $z'=(0,\cdots,0,z,0,\cdots,0)$ we have $Q_{z'}=G\times Q^{n-1}$, which is also non-standard, contradiction. Hence, all isotropy subgroups of $\rho^{+}_j$ must be standard, for all $j$.

Since these subgroups are standard, Lemma \ref{lemma:Iso(p)=Iso(p_CC)} implies that  
$$Iso(\rho)=Iso(\rho_{+})$$

Let $Q_i^{n-1}$ be a maximal isotropy subgroup of $\rho^{+}$. Corollary \ref{cor:conjugation isotropies} implies that, up to conjugation, there exist $\rho_i^{+}$ with isotropy subgroup $Q_i^{n-1}$. So that
$$Iso(\rho_i^{+})=Q_i^{n-1}$$
for $1\leq i\leq k$.
Since all the $Q_i^{n-1}$, $i=1,\cdots,n$ appear, we obtain $k=n$. Therefore,
$$\rho^{+}=\rho^{+}_1\oplus\cdots\oplus\rho^{+}_n$$
Hence it turns out that 
$$\rho=\rho_1\oplus\cdots\oplus\rho_n$$
because the categories $Rep(Q^n,\RR)=Rep^{+}(Q^n,\CC)$ and $Rep_{\RR}(Q^n,\CC)$ are equivalent.

On the other hand, assume that the representation $\rho$ comes from $Rep^{-}(Q^n,\CC)$. In this case, we consider the realification of $\rho$, i.e. by restriction of scalars from $\CC$ to $\RR$ we have the real representation $\rho_{\RR}=\rho$ of $Q^n$. Again, the hypothesis implies that 
$$\rho_{\RR}=\rho_{1\RR}\oplus\cdots\oplus\rho_{k\RR}$$
with maximal standard isotropy subgroups conjugate to $Q^{n-1}$. So that 
$$Iso(\rho)=Iso(\rho_{\RR})$$
As before, consider $Q^{n-1}_i$, the isotropy subgroup of $\rho_{\RR}$. Then, Corollary \ref{cor:conjugation isotropies} and the uniqueness of $Q^{n-1}_i$ for each $i$ show that 
$$\rho=\rho_1\oplus\cdots\oplus\rho_n$$

In both cases, each $\rho_i$ is an irreducible complex representation, that has maximal standard isotropy subgroups.
 
Proposition \ref{prop:at_most2} implies that 
$$\rho_i=\rho_{i_1}\otimes\cdots\otimes\rho_{i_n}$$
where the tensor product representations $\rho_{i_j}$'s are trivial except at most two of them, which are of real dimension $4$. So each irreducible representation of $Q^n$ in the decomposition must be of real dimension $4$. Thus we have $\rho=\rho_1\oplus\cdots\oplus\rho_n$, where $\rho_i=\rho_{i_1}\otimes\cdots\otimes\rho_{i_n}$. The $n's$ are all the same here.
\end{proof}

\section{Rigidity}

Let $M^{4n}$ be a locally regular $Q^n$-manifold and $X = M/Q^n$ the quotient, which is a nice manifold with corners. In this section, we assume that $X$ is a homotopy polytope so that all the faces as well $X$ itself are contractible manifolds with boundary. Then $M^{4n}\cong_{Q^n} \mathcal{M}(\lambda)$, for a characteristic function $\lambda$, covering the identity on $X$.

\begin{definition}
Let $N$ be a $G$-manifold. The $G$-action is called \emph{locally linear} if, for each $y \in N$ with isotropy group $G_y$, there is a $G_y$-neighborhood of $y$, a $G_y$-representation $V$ and a open $G$-embedding
$$G{\times}_{G_y}V \to N,$$
called a \emph{linear tube}.
\end{definition}

For details on locally linear actions see \cite{bredon}.

\begin{remark}
\begin{enumerate}
\item[(1)] The $Q^n$-action on $M^{4n}$ is effective.
\item[(2)] The isotropy groups of the $Q^n$-action on $M^{4n}$ are standard subgroups of $Q^n$.
\item[(3)] The $Q^n$-action on $M^{4n}$ is locally linear.
\end{enumerate} 
\end{remark}

Let $N^{4n}$ be a locally linear $Q^n$-manifold and $f: N^{4n} \to M^{4n}$ be a $Q^n$-homotopy equivalence. We will prove the results of the remark for $N^{4n}$, following closely \cite{mp}.

\begin{lemma}
The $Q^n$ action on $N^{4n}$ is effective.
\end{lemma}
\begin{proof}
We assume that this is not the case. So there is some $q\in Q^n$ that fixes $N^{4n}$ pointwise. Let $G=<q>$. Then $N^G=N^{4n}\simeq M^G$ since $f$ is an equivariant homotopy equivalence. But $M^G$ is a closed proper submanifold of $M^{4n}$, because the action on $M^{4n}$ is effective. Thus, $dim(N^G)=dim(M^G)<dim(M^{4n})=dim(N^{4n})$, contradiction.
\end{proof}

\begin{lemma}
The non-trivial isotropy subgroups of $N^{4n}$ are standard.
\end{lemma}
\begin{proof}
Let $z$ be in $N^{4n}$ with isotropy group $Q_z$, that is not (canonical) standard. Since $N^{Q_z}\simeq M^{Q_z}$ and $N^{Q_z}\neq \emptyset$, we have that $M^{Q_z}\neq \emptyset$. Since non-trivial isotropy subgroups of $M^{4n}$ are (canonical) standard subgroups $Q^k$, $M^{Q'}=M^{Q_z}$ for some $Q'$ that strictly contains $Q_z$. But $z\in N^{Q_z}$ and $z\notin N^{Q'}$. Thus $N^{Q_z}\supsetneq N^{Q'}\simeq M^{Q'}$. Since fixed points are closed submanifolds without boundary, we have that $dimM^{Q_z}=dimN^{Q_z}>dimN^{Q'}=dimM^{Q'}$. But this is a contradiction since $dimM^{Q_z}=dimM^{Q'}$.
\end{proof}

\begin{corollary} \label{cor:sameISO}
The isotropy subgroups of $N^{4n}$ and $M^{4n}$ are the same.
\end{corollary}

\begin{corollary} \label{cor:nonempty_fixset}
For each isotropy group $Q'$, each component of the fixed point set $N^{Q'}$ contains a $Q^n$-fixed point.
\end{corollary}
\begin{proof}
Since the isotropy subgroups of $M^{4n}$ and $N^{4n}$ are the same, the subgroup $Q'$ is equal to $Q_F$ for some face $F$ of $X$. Then $N^{Q'} = N^{Q_F}$ and let $C$ be a component of $N^{Q_F}$. The map $f$ induces a homotopy equivalence $f^{Q_F}:  N^{Q_F} \to M^{Q_F}$. Then the restriction of $f^{Q_F}$ to $C$ induces a homotopy equivalence from $C$ to $C'$ where $C'$ is a component of $M^{Q_F}$. This implies that $C'$ has a $Q^n$-fixed point. Therefore $C$ has a $Q^n$-fixed point.
\end{proof}

\begin{proposition} \label{prop:loc_regular}
The $Q^n$-action on $N^{4n}$ is locally regular. Furthermore, $N^{4n}/Q^n=Y$ is a homotopy polytope.
\end{proposition}
\begin{proof}
The proof is as in \cite[Proposition 7.7]{mp}. Let $y\in N^{4n}$ with isotropy group $Q_y$. Corollary \ref{cor:sameISO} implies that $Q_y=Q_F$ for some face $F<X$. Let $C$ be that component of $N^{Q_F}$ that contains $y$. Corollary \ref{cor:nonempty_fixset} implies that there is a point $z\in C$ that is fixed by $Q^n$. Since the action is locally linear, there is a linear slice around each point. Since the isotropy groups are standard, the slices are well defined up to linear equivalence \cite{schultz}. Notice that the tubes are $Q^n$-manifolds. The slice $S$ is an effective linear representation of $Q^n$. If $y\in S$, then by Theorem \ref{thm:isotropy gps} the action of $Q^n$ on $S$ is modeled by the regular action on $\HH^n$. So $S$ is a regular neighborhood of $y$. In general, there is a path $\alpha$ that joins $z$ to $y$. We start with a finite open cover of the image of $\alpha$ with tubes. Then, there is a tube $\tau$ that contains $z$. The tangent space $\mathcal{T}_z$ is a $Q^n$-representation and therefore, it is a representation coming from a regular action. Let $y_0$ be a point on the path that lies in the intersection of $\tau$ and another tube at $\tau_0$ at $x_0$. Then the tangent space at $y_0$ is a $\mathcal{T}_{y_0}$-representation, and this is the restriction of the $Q$-representation $\mathcal{T}_z$ and the $Q_{x_0}$-representation $\mathcal{T}_{x_0}$. Continuing this way we get a sequence of points $y_i$, $i=0,1,\cdots ,k$ such that $y_k=y$ and the $Q_{y_i}$-representation $\mathcal{T}_{y_i}$ is the restriction of the $Q$-representation $\mathcal{T}_z$. That completes the proof of the first part.

The quotient $Y$ is a nice n-manifold with corners and there is a skeletal homotopy equivalence $\phi: Y\to X$. This implies that $Y$ is a homotopy polytope \cite[Proposition 3.16]{mp}. 
\end{proof}

As in \cite{mp}, \cite{ps}, we can show the following. 

\begin{lemma} \label{lemma:Poincare}
Let $X$ and $Y$ be homotopy polyhedra. Then any face-preserving homotopy equivalence ${\phi}: Y \to X$ is face-preserving homotopic to a face- preserving homeomorphism. Furthermore, if $(Y, {\lambda}')$ and $(X, {\lambda})$ are characteristic pairs, the induced map ${\phi}_*: N_Y({\lambda}') \to M_X({\lambda})$ is $Q^n$-homotopic to a $Q^n$-homeomorphsim.
\end{lemma}
\begin{proof}
The proof is done by induction. For the inductive step, we use the surgery exact sequence and the \Poincare Conjecture. For details see \cite{mp}, \cite{ps}.

The second part is immediate from the constructions.
\end{proof}

\begin{theorem}[Rigidity of Quoric Manifolds] Let $M^{4n}$ be a locally linear quoric manifold over a nice $n$-manifold with corners $X$ and characteristic map $\lambda$. 
We assume that $X$ is a homotopy polytope and all the faces of $X$ (and $X$ itself) are contractible manifolds with corners. Let $f: N^{4n} \to M^{4n}$ be a $Q^n$-homotopy equivalence, where $N^{4n}$ is a locally linear manifold. Then $f$ is $Q^n$-homotopic to a $Q^n$-homeomorphism.
\end{theorem}
\begin{proof}
By Proposition \ref{prop:loc_regular}, the action on $N^{4n}$ is locally regular with quotient $Y$. 
%By the homotopy invariance of the Euler class 
%(Theorem \ref{thm:HomInvarCech}), $e(N) = e(M) = 0$. 
Thus, by Corollary \ref{cor:topclas2}, $N \cong_{Q^n} N_Y({\lambda}')$, for some characteristic map ${\lambda}'$. It enough to show that the map
$$\bar{f}: N_Y({\lambda}') \xrightarrow{\cong} N \xrightarrow{f} M \xrightarrow{\cong} M_X({\lambda})$$
is $Q^n$-homotopic to a $Q^n$-homeomorphism. Notice that the map $\bar{f}$ induces a face-preserving homotopy equivalence 
$\phi : Y \to X$. From Proposition \ref{prop:reverse_construction}, $\bar{f} \simeq_{Q^n} {\phi}_*$. From Lemma \ref{lemma:Poincare}, $\phi$ is face-preserving homotopic to a face-preserving homeomorphism and thus ${\phi}_*$ is $Q^n$-homotopic to a $Q^n$-homeomorphism, completing the proof.
\end{proof}

\subsection*{Acknowledgements}
The authors want to thank Michael Davis for suggesting the problem and helping with useful discussions and suggestions. They also would like to thank 
Tadeusz Januszkiewicz whose comments considerably improve the presentation of the paper. 
The first author  would like to thank Vassilis Metaftsis, Ioannis Emmanouil and Takis Hatzinikitas for their feedback on the first author's thesis at the University of the Aegean, on which the paper is based.

%%%%%%%%%%% To ease editing, use normal size for the references:

\end{document}